\documentclass[a4paper, 11pt, reqno]{amsart}
\usepackage[varg]{txfonts}
\usepackage{verbatim}
\usepackage{amsfonts, amsmath, amssymb, amsthm, enumerate, color}

\newcommand{\ep}{\varepsilon}
\newcommand{\R}{\mathbb{R}}

\newtheorem{thm}{Theorem}[section]
\newtheorem{lemma}{Lemma}[section]

\newtheorem{prop}{Proposition}[section]
\newtheorem{cor}{Corollary}[section]
\newtheorem{remark}{Remark}[section]

\author[J. Seok]{Jinmyoung Seok}
\address[Jinmyoung Seok]{Department of Mathematics, Kyonggi University,
154-42 Gwanggyosan-ro, Yeongtong-gu, Suwon 16227, Republic of Korea}
\email{jmseok@kgu.ac.kr}

\title[Limit profiles and uniqueness of the nonlinear Choquard equations]
{Limit profiles and uniqueness of ground states \\ to the nonlinear Choquard equations}

\begin{document}

\maketitle

\begin{abstract}
Consider nonlinear Choquard equations
\begin{equation*}
\left\{\begin{array}{rcl}
-\Delta u +u & = &(I_\alpha*|u|^p)|u|^{p-2}u  \quad \text{in } \R^N, \\
\lim_{x \to \infty}u(x) & = &0,
\end{array}\right.
\end{equation*} 
where $I_\alpha$ denotes Riesz potential and $\alpha \in (0, N)$. 
In this paper, we investigate limit profiles of ground states of nonlinear Choquard equations as $\alpha \to 0$ or $\alpha \to N$. 
This leads to the uniqueness and nondegeneracy of ground states when $\alpha$ is sufficiently close to $0$ or close to $N$.
\end{abstract}

\section{Introduction}
Let $N \geq 3,\, \alpha \in (0, N)$, $p > 1$.
We are concerned with the so-called nonlinear Choquard equation:
\begin{equation}\label{main-eq}
\left\{\begin{array}{rcl}
-\Delta u +u & = &(I_\alpha*|u|^p)|u|^{p-2}u  \quad \text{in } \R^N, \\
\lim_{x \to \infty}u(x) & = &0
\end{array}\right.
\end{equation}
where $I_\alpha$ is Riesz potential given by
\[
I_\alpha(x) = \frac{\Gamma(\frac{N-\alpha}{2})}{\Gamma(\frac{\alpha}{2})\pi^{N/2}2^\alpha|x|^{N-\alpha}}
\]
and $\Gamma$ denotes the Gamma function.
The equation \eqref{main-eq} finds its physical origin especially when $N=3, \alpha = 2$ and $p = 2$. 
In this case, a solution of the equation
\begin{equation}\label{main-eq-physical}
-\Delta u +u  = (I_2*|u|^2)u
\end{equation}
gives a solitary wave of the Schr\"odinger type  nonlinear  evolution equation 
\[
i\partial_t\psi +\Delta\psi + (I_2*|\psi|^2)\psi = 0,
\] 
which describes, through Hartree-Fock approximation, a dynamics of condensed states to a system of nonrelativistic bosonic particles with two-body attractive interaction potential $I_2$ that is Newtonian potential \cite{FL, LNR}.
The equation \eqref{main-eq-physical} also arises as a model of a polaron by Pekar \cite{P} or in an approximation of Hartree-Fock theory for a one-component plasma \cite{L}.

The equation \eqref{main-eq} enjoys a variational structure. It is the Euler-Lagrange equation of the functional
\[
J_\alpha(u) = \frac12\int_{\R^N}|\nabla u|^2+u^2 \,dx -\frac{1}{2p}\int_{\R^N}(I_\alpha*|u|^p)|u|^p\,dx.
\]
From Hardy-Littlewood-Sobolev inequality (Proposition \ref{HLS} below), one can see that $J_\alpha$ is well defined and is continuously differentiable
on $H^1(\R^N)$ if $p \in [1+\frac{\alpha}{N},\, \frac{N+\alpha}{N-2}]$.
We say a function $u \in H^1(\R^N)$ is a ground state solution to \eqref{main-eq} if $J_\alpha'(u) = 0$ and
\[
J_\alpha(u) = \inf\left\{ J_\alpha(v)  ~|~ v \in H^1(\R^N),\, J_\alpha'(v) = 0,\, v \not\equiv 0 \right\}.
\]
When $N = 3,\alpha =2$ and $p = 2$, the existence of a radial positive solution is proved in \cite{L, Li, M} by variational methods and in \cite{CSV, MPT, TM} by ODE approaches. 
In \cite{MV2}, Moroz and Van Schaftingen prove the existence of a ground state solution to \eqref{main-eq} in the range of $p \in (1+\frac{\alpha}{N},\, \frac{N+\alpha}{N-2})$
and the nonexistence of a nontrivial finite energy solution of \eqref{main-eq} for $p$ outside of the above range. 
For qualitative properties of ground states to \eqref{main-eq}, we refer to \cite{MZ, MV2}.

In this paper, we are interested in limit behaviors of ground state to \eqref{main-eq} as either $\alpha \to 0$ or $\alpha \to N$.
These shall play essential roles to prove the uniqueness and nondegeneracy of a positive radial ground state to \eqref{main-eq} for $\alpha$ sufficiently close to $0$ or $N$. 
From the existence results by Moroz and Van Schaftingen, we can see that a positive radial ground state of \eqref{main-eq} exists for every $\alpha \in (0,\, N(p-1))$ when $p \in (1,\, N/(N-2))$ is fixed. 
Also for given $p \in (2,\, 2N/(N-2))$, a positive radial ground state of \eqref{main-eq} exists for every $\alpha \in ((N-2)p-N,\, N)$.

As $\alpha \to 0$, one is possible to see that the functional $J_\alpha$ formally approaches to 
\[
J_0(u) :=  \frac12\int_{\R^N}|\nabla u|^2+u^2 \,dx -\frac{1}{2p}\int_{\R^N}|u|^{2p}\,dx \quad \text{on } H^1(\R^N)
\]
because $I_\alpha * f$ approaches to $f$ as $\alpha\to 0$.
It is well known that the Euler-Lagrange equation (equation \eqref{limit-eq0} below) of $J_0$ admits a unique positive radial ground state solution. 
Thus it is reasonable to expect that the ground state of \eqref{limit-eq0} is the limit profile of ground states of \eqref{main-eq} as $\alpha\to0$. 
Our first result is to confirm this. 
\begin{thm}\label{profile0}
Fix $p \in (1,\, N/(N-2))$. Let $\{u_\alpha\}$ be a family of positive radial ground states to \eqref{main-eq} for $\alpha$ close to $0$ and
$u_0$ be a unique positive radial ground state of the equation
\begin{equation}\label{limit-eq0}
\left\{\begin{array}{rcl}
-\Delta u +u & = &|u|^{2p-2}u \quad \text{in } \R^N, \\
\lim_{x \to \infty}u(x) & = &0.
\end{array}\right.
\end{equation}
Then one has
\[
\lim_{\alpha\to0}\|u_\alpha-u_0\|_{H^1(\R^N)} = 0.
\]
%In addition, there exist constants $C, c > 0$, which is independent of $\alpha$ close to $0$ such that
%\[
%u_\alpha(x) \leq Ce^{-c|x|}.
%\]
\end{thm}
%\begin{remark}

%\end{remark}

On the other hand, the functional $J_\alpha$ blows up when $\alpha \to N$ due to the term $\Gamma((N-\alpha)/2)$ in the coefficient of $I_\alpha$. 
Thus we need to get rid of this by taking a scaling $v = s(N,\alpha,p)u$ where
\[
s(N,\alpha, p) := \left(\frac{\Gamma(\frac{N-\alpha}{2})}{\Gamma(\frac{\alpha}{2})\pi^{N/2}2^\alpha}\right)^{\frac{1}{2p-2}}
\sim \quad \left(\frac{1}{N-\alpha}\right)^{\frac{1}{2p-2}} \quad \text{as } \alpha \to N.
\]
With this scaling, $J_\alpha$ transforms into the following functional which we still denote by $J_\alpha$ for simplicity,
\[
J_\alpha(v) = \frac12\int_{\R^N}|\nabla v|^2+v^2 \,dx -\frac{1}{2p}\int_{\R^N}\left(\frac{1}{|\cdot|^{N-\alpha}}*|v|^p\right)|v|^p\,dx.
\]

Then as $\alpha \to N$, $J_\alpha$ approaches to 
\[
J_N(v) = \frac12\int_{\R^N}|\nabla v|^2+v^2 \,dx -\frac{1}{2p}\left(\int_{\R^N}|v|^p\,dx\right)^2.
\]
It is easy to see that for $p \in (2,\, 2N/(N-2))$, the limit functional $J_N$ is $C^1$ on $H^1(\R^N)$  and its Euler-Lagrange equation is
\begin{equation}\label{limit-eqN}
\left\{\begin{array}{rcl}
-\Delta v +v & = &\left(\int_{\R^N}|v|^p\,dx\right)|v|^{p-2}v \quad \text{in } \R^N, \\
\lim_{x \to \infty}v(x) & = &0.
\end{array}\right.
\end{equation}
{
The existence and properties of a ground state to \eqref{limit-eqN} is studied in \cite{RV}.
More precisely, it is shown in \cite{RV} that there exists a positive radial ground state $v_0$ of the equation \eqref{limit-eqN}.
Furthermore the following properties for ground states to \eqref{limit-eqN} are proved:
\begin{enumerate}[$(i)$]
\item the ground state energy level of \eqref{limit-eqN} satisfies the mountain pass characterization, i.e,
\[
J_N(v_0) = \min_{v \in H^1(\R^N) \setminus \{0\}}\max_{t\geq 0}J_N(tv);
\]
\item any ground state of \eqref{limit-eqN} is sign-definite, radially symmetric up to a translation and strictly decreasing in radial direction.
\item any ground state of \eqref{limit-eqN} decays exponentially as $|x| \to \infty$.
\end{enumerate}
}

Our next result establishes uniqueness and linearized nondegeneracy of the ground state $v_0$ of \eqref{limit-eqN}.
\begin{thm}\label{prop-limit-eqN}
For $p \in (2,\, 2N/(N-2))$, let  $v_0$ be a positive radial ground state to \eqref{limit-eqN}.
Then we have the following: 
\begin{enumerate}[$(i)$]
\item there is no any other positive radial ground state to \eqref{limit-eqN};
\item the linearized equation of \eqref{limit-eqN} at $v_0$, given by
\begin{equation}\label{linearized-limit-eqN}
-\Delta \phi +\phi -p\left(\int_{\R^N}v_0^{p-1}\phi\,dx\right)v_0^{p-1}-(p-1)\left(\int_{\R^N}v_0^p\,dx\right)v_0^{p-2}\phi  = 0 \quad \text{in } \R^N,
\end{equation}
only admits { solutions of the form
\[
\phi = \sum_{i=1}^N c_i\partial_{x_i}v_0, \qquad c_i \in \R
\]
in the space $L^2(\R^N)$.}  
\end{enumerate}
\end{thm}
Using uniqueness of $v_0$, we can obtain an analogous result to Theorem \ref{profile0}.
\begin{thm}\label{profileN}
Fix $p \in (2,\, 2N/(N-2))$. Let $\{u_\alpha\}$ be a family of positive radial ground states to \eqref{main-eq} for $\alpha$ close to $N$ and
$v_0 \in H^1(\R^N)$ be a unique positive radial ground state of \eqref{limit-eqN}.
Then one has
\[
\lim_{\alpha \to N}\|v_\alpha-v_0\|_{H^1(\R^N)} = 0
\]
where $v_\alpha$ is a family of rescaled functions given by $v_\alpha := s(N,\alpha, p) u_\alpha$.
\end{thm}
\begin{remark}
By applying the standard comparison principle, 
it is also possible to see that there exist constants $C, c > 0$ which is independent of $\alpha$ close to $N$ such that
\[
u_{\alpha}(x) \leq C(N-\alpha)^{\frac{1}{2p-2}}e^{-c|x|},
\]
which shows the vanishing profiles of $u_\alpha$.
\end{remark}

The limit profiles of ground states to \eqref{main-eq} lead to the uniqueness and nondegeneracy of them for $\alpha$ either close to $0$ or close to $N$. 
When $N=3,\, \alpha = 2$ and $p = 2$, these are proved by Lenzmann \cite{Lenzmann} and Wei-Winter \cite{WW}. 
Xiang \cite{Xiang} extends this result to the case that $N=3,\, \alpha = 2$ and $p > 2$ close to $2$ by using perturbation arguments. 

We say a positive radial ground state $u_\alpha$ of \eqref{main-eq} is nondegenerate if the linearized equation of \eqref{main-eq} at $u_\alpha$, given by
\[
-\Delta \phi +\phi -p(I_\alpha * (u_\alpha^{p-1}\phi))u_\alpha^{p-1}-(p-1)(I_\alpha * u_\alpha^p)u_\alpha^{p-2}\phi  = 0 \quad \text{in } \R^N,
\]
only admits solutions of the form
\[
\phi = \sum_{i=1}^N c_i\partial_{x_i}u_\alpha, \qquad c_i \in \R
\]
in the space $L^2(\R^N)$.  
We should assume $p \geq 2$ for the well-definedness of the linearized equation. 
\begin{thm}[Uniqueness and nondegeneracy]\label{thm1}
Fix $p \in [2,\, \frac{N}{N-2})$. Then a positive radial ground state of \eqref{main-eq} is unique and nondegenerate for $\alpha$ sufficiently close to $0$.
Fix $p \in (2,\, \frac{2N}{N-2})$. Then the same conclusion holds true for $\alpha$ sufficiently close to $N$.
\end{thm}
\begin{remark}
Here we note that in the case that $\alpha$ is close to $0$, the uniqueness and nondegeneracy are proved only when $N = 3$
but in the case that $\alpha$ is close to $N$, these are proved for every dimension $N \geq 3$.  
\end{remark}

%\begin{remark}
It is worth mentioning that unlike the family of ground states $u_\alpha$ to \eqref{main-eq}, the family of least energy nodal solutions $\tilde{u}_\alpha$ to \eqref{main-eq} (the minimal energy solution among the all nodal solutions) does not converge to any nontrivial solution of the limit equations \eqref{limit-eq0} or \eqref{limit-eqN}, even up to a translation and up to a subsequence.
Actually, the asymptotic profile of $\tilde{u}_\alpha$ is shown to be $u_0(\cdot-\xi^+_\alpha) -u_0(\cdot-\xi^-_\alpha)$ as $\alpha \to 0$ 
and $(N-\alpha)^{\frac{1}{2p-2}}(v_0(\cdot-\xi^+_\alpha) -v_0(\cdot-\xi^-_\alpha))$ as $\alpha \to N$
for some $\xi^+_\alpha,\, \xi^-_\alpha \in \R^N$ such that $\lim_{\alpha \to 0}|\xi^+_\alpha-\xi^-_\alpha| = 0$.
See \cite{RV} for the proof. 
Relying on this fact and the nondegeneracy of the ground state $u_0$ to \eqref{limit-eq0},  
it is also proved in \cite{RV} that $\tilde{u}_\alpha$ is odd-symmetric with respect to the hyperplane normal to the vector $\xi^+_\alpha-\xi^-_\alpha$ and through the point $(\xi^+_\alpha-\xi^-_\alpha)/2$ when $\alpha \sim 0$ or $\alpha \sim N$. 
%\end{remark}

The rest of this paper is organized as follows. In Section 2, we collect some useful auxiliary tools
and technical results which are frequently invoked when proving the main theorems. 
Theorem \ref{profile0} is proved in Section 3. Theorem \ref{prop-limit-eqN} and \ref{profileN} are proved in Section 4.
In subsequent sections, we prove our uniqueness and nondegeneracy results respectively.

\section{Auxiliary results}
In this section, we provide with some useful known results and auxiliary tools.
We begin with giving sharp information on the best constant of Hardy-Littlewood-Sobolev inequality. This plays an important role in our analysis. 
\begin{prop}[Hardy-Littlewood-Sobolev inequality \cite{Gra, LL}]\label{HLS}
Let $p, r > 1$ and $0 < \alpha < N$ be such that
\[
\frac{1}{p}+\frac{1}{r} = 1+\frac{\alpha}{N}.
\]
Then for any $f \in L^p(\R^N)$ and $g \in L^r(\R^N)$, one has
\[
\left|\int_{\R^N}\!\!\int_{\R^N}\frac{f(x)g(y)}{|x-y|^{N-\alpha}}\,dxdy\right| \leq C(N,\alpha,p)\|f\|_{L^p(\R^N)}\|g\|_{L^r(\R^N)}.
\]
The sharp constant satisfies
\[
C(N,\alpha,p) \leq \frac{N}{\alpha}\left(|\mathbb{S}|^{N-1}/N\right)^{\frac{N-\alpha}{N}}\frac{1}{pr}
\left(\left(\frac{(N-\alpha)/N}{1-1/p}\right)^{\frac{N-\alpha}{N}}+\left(\frac{(N-\alpha)/N}{1-1/r}\right)^{\frac{N-\alpha}{N}}\right)
\]
where $|\mathbb{S}^{N-1}|$ denotes the surface area of the $N-1$ dimensional unit sphere. 

In addition, if $p = r = \frac{2N}{N+\alpha}$, then
\[
C\left(N, \alpha, \frac{2N}{N+\alpha}\right) = \pi^{\frac{N-\alpha}{2}}\frac{\Gamma(\alpha/2)}{\Gamma((N+\alpha)/2)}\left(\frac{\Gamma(N)}{\Gamma(N/2)}\right)^{\frac{\alpha}{N}}.
\]
\end{prop}

The following Riesz potential estimate is equivalent to Hardy-Littlewood-Sobolev inequality.
\begin{prop}[\cite{Gra, LL}]\label{est-riesz}
Let $1 \leq r < s < \infty$ and $0 < \alpha < N$ be such that
\[
\frac1r-\frac1s = \frac{\alpha}{N}.
\]
Then for any $f \in L^r(\R^N)$, one has
\[
\left\|\frac{1}{|\cdot|^{N-\alpha}}*f\right\|_{L^s(\R^N)} \leq K(N,\alpha,r) \|f\|_{L^r(\R^N)}.
\]
Here, the sharp constant $K(N,\alpha,r)$ satisfies
\[
\limsup_{\alpha\to0}\alpha\, K(N,\alpha,r) \leq \frac{2}{r(r-1)}|{\mathbb{S}}^{N-1}|.
\]
\end{prop}

\begin{cor}\label{est-riesz-uniform}
Let $r, s$ satisfy the assumption in Proposition \ref{est-riesz}
Then for small $\alpha > 0$, there exists $C = C(N,r) > 0$ such that for any $f \in L^r(\R^N)$, 
\[
\left\|I_\alpha*f\right\|_{L^s(\R^N)} \leq C\|f\|_{L^r(\R^N)}.
\]
\end{cor}
\begin{proof}
This immediately follows from Proposition \ref{est-riesz} and the fact $\Gamma(\alpha/2) \sim 1/\alpha$ as $\alpha \to 0$.
\end{proof}

We denote by $H^1_r(\R^N)$ the space of radial functions in $H^1(\R^N)$.
The following compact embedding result is proved in \cite{Strauss}.
\begin{prop}
The Sobolev embedding $H^1_r(\R^N) \hookrightarrow L^p(\R^N)$ is compact if $2 < p < 2N/(N-2)$.
\end{prop}

By combining Hardy-Littlewood-Sobolev inequality (Proposition \ref{HLS}) and the compact Sobolev embedding, it is easy to see the following convergence holds. 
\begin{prop}\label{convergence-simpler}
Let $\alpha \in (0,\, N)$ and $p \in (1+\frac{\alpha}{N},\, \frac{N+\alpha}{N-2})$ be given.
{Let} $\{u_j\} \subset H_r^1(\R^N)$ be a sequence converging weakly to some $u_0 \in H^1_r(\R^N)$ in $H^1(\R^N)$ as $j \to \infty$,
then
\[
\int_{\R^N}\left(\frac{1}{|\cdot|^{N-\alpha}}*|u_j|^p\right)|u_j|^p\,dx 
\to \int_{\R^N}\left(\frac{1}{|\cdot|^{N-\alpha}}*|u_0|^p\right)|u_0|^p\,dx.
\]
In addition, for any $\phi \in H^1(\R^N)$, 
\[
\int_{\R^N}\left(\frac{1}{|\cdot|^{N-\alpha}}*|u_j|^p\right)|u_j|^{p-2}u_j\phi\,dx 
\to \int_{\R^N}\left(\frac{1}{|\cdot|^{N-\alpha}}*|u_0|^p\right)|u_0|^{p-2}u_0\phi\,dx.
\]
\end{prop}

It is useful to obtain estimates for $I_\alpha * (|u|^{p-1}u\phi)$ as $\alpha \to 0$ and $(1/|\cdot|^{N-\alpha})*(|u|^{p-1}u\phi)$ as $\alpha \to N$ when $u, \phi \in H^1(\R^N)$. 
\begin{prop}\label{est-identity}
Let $u, \phi \in H^1(\R^N)$. Then one has the following:
\begin{enumerate}[$(i)$] 
\item  For every $1 < p < \frac{N}{N-2}$ and $0 < \alpha < \frac{p-1}{2}N$, there exists $C = C(N) > 0$ independent of $\alpha$ near $0$ such that
\[
\|I_\alpha*(|u|^{p-2}u\phi)\|_{L^2(\R^N)} < C\|u\|_{H^1(\R^N)}^{p-1}\|\phi\|_{H^1(\R^N)}
\]
and
\[
\lim_{\alpha \to 0}\|~I_\alpha*(|u|^{p-2}u\phi)-|u|^{p-2}u\phi~\|_{L^2(\R^N)} =0;
\]
\item For every $2 < p < \frac{2N}{N-2}$ and $\frac{(N-2)p}{2} < \alpha < N$, there exists $C = C(N,p) > 0$ independent of $\alpha$ near $N$ such that
\[
\left\|~\frac{1}{|\cdot|^{N-\alpha}}*(|u|^{p-2}u\phi)~\right\|_{L^\infty(\R^N)} < C\|u\|_{H^1(\R^N)}^{p-1}\|\phi\|_{H^1(\R^N)}
\] 
and 
\[
\lim_{\alpha \to N}\left\|~\left(\frac{1}{|\cdot|^{N-\alpha}}*(|u|^{p-2}u\phi)\right)-\int_{\R^N}|u|^{p-2}u\phi\,dx~\right\|_{L^\infty(K)} =0
\]
for any compact set $K \subset \R^N$.
\end{enumerate}
\end{prop}
\begin{proof}
A proof for $(i)$ can be found in \cite{S}. We prove $(ii)$. Observe from the H\"older inequality that
\[
\begin{aligned}
&\left(\frac{1}{|\cdot|^{N-\alpha}}*|u|^{p-2}u\phi\right)(x) \\
&\leq \int_{B_1(x)}\frac{1}{|x-y|^{N-\alpha}}(|u|^{p-2}u\phi)(y)\,dy + \int_{B_1^c(x)}\frac{1}{|x-y|^{N-\alpha}}(|u|^{p-2}u\phi)(y)\,dy \\
& \leq\left(\int_{B_1(0)}\frac{1}{|y|^{(N-\alpha)\frac{2^*}{2^*-p}}}\,dy\right)^{\frac{2^*-p}{2^*}}\|u\|_{L^{2^*}(\R^N)}^{p-1}\|\phi\|_{L^{2^*}(\R^N)}
+\|u\|_{L^p(\R^N)}^{p-1}\|\phi\|_{L^p(\R^N)},
\end{aligned}
\]
where $2^*$ denotes the critical Sobolev exponent $2N/(N-2)$. Note from the condition on $\alpha$ that $(N-\alpha)\frac{2^*}{2^*-p} < N$ so that
the integral $\int_{B_1(0)}1/|y|^{(N-\alpha)\frac{2^*}{2^*-p}}\,dy$ is uniformly bounded for $\alpha$ sufficiently close to $N$.
This proves the former assertion of $(ii)$.
To prove the latter, we suppose the contrary. Then, there exist a compact set $K$ and sequences $\alpha_j \to N$, $x_j \in K$ such that
\begin{equation}\label{eq-2-1}
\int_{\R^N}\frac{1}{|x_j-y|^{N-\alpha_j}}(|u|^{p-2}u\phi)(y)\,dy \not\to \int_{\R^N}|u|^{p-2}u\phi\,dy \quad \text{as } j \to \infty.
\end{equation}
Define $f_j(y) = (|u|^{p-2}u\phi)(y)/|x_j-y|^{N-\alpha_j}$ so that as $j \to \infty$, $f_j(y) \to (|u|^{p-2}\phi)(y)$ almost everywhere.  
We may assume $x_j \to x_0$ as $j\to\infty$ for some $x_0 \in K$.
We claim that $f_j$ is uniformly integrable and tight in $\R^N$, i.e., for given $\ep > 0$, there exists $\delta > 0$ such that
$\int_{E}|f_j(y)|\,dy < \ep$ for every $E \subset \R^N$ satisfying $|E| < \delta$ and there exists $R > 0$ such that $\int_{B^c_R(0)}|f_j(y)|\,dy < \ep$.
Indeed, we have
\[
\begin{aligned}
\int_{E}|f_j(y)|\,dy &\leq \left(\int_{E}\frac{1}{|x_j-y|^{(N-\alpha_j)\frac{2^*}{2^*-p}}}\,dy\right)^{\frac{2^*-p}{2^*}}\|u\|_{L^{2^*}(E)}^{p-1}\|\phi\|_{L^{2^*}(E)} \\
&= \left( \left(\int_{E \cap B_1(x)} +\int_{E \cap B^c_1(x)} \right)  \frac{1}{|x_j-y|^{(N-\alpha_j)\frac{2^*}{2^*-p}}}\,dy\right)^{\frac{2^*-p}{2^*}}\|u\|_{L^{2^*}(E)}^{p-1}\|\phi\|_{L^{2^*}(E)} \\
&\leq \left(\int_{B_1(0)}\frac{1}{|y|^{(N-\alpha_j)\frac{2^*}{2^*-p}}}\,dy+|E| \right)^{\frac{2^*-p}{2^*}}\|u\|_{L^{2^*}(E)}^{p-1}\|\phi\|_{L^{2^*}(E)},
\end{aligned}
\]
which shows $f_j$ is uniformly integrable. Take also a large $R > 0$ such that $B_2(x_0) \subset B_R(0)$. Then
since $\int_{B^c_R(0)}|f_j(y)|\,dy \leq \int_{B^c_R(0)}|u(y)|^{p-1}|\phi(y)|\,dy$, the tightness of $f_j$ is proved. 
Now the Vitali convergence theorem says that $\int_{\R^N} f_j(y)\,dy \to \int_{\R^N}|u|^p(y)\,dy$, which contradicts with \eqref{eq-2-1}.
This completes the proof. 
\end{proof}

\begin{prop}\label{convergence0}
Fix $1 < p < \frac{N}{N-2}$. Let $\{\alpha_j\} > 0$ be a sequence converging to $0$ and
$\{u_j\}\subset H_r^1(\R^N)$ be a sequence converging weakly in $H^1(\R^N)$ to some $u_0 \in H^1_r(\R^N)$. 
Then as $j \to \infty$ the following holds:
\begin{enumerate}[$(i)$]
\item
$\displaystyle\int_{\R^N}(I_{\alpha_j}*|u_j|^p)|u_j|^p\,dx \to \int_{\R^N}|u_0|^{2p}\,dx$;

\item for any $\phi \in H^1(\R^N)$,
\[
\int_{\R^N}(I_{\alpha_j}*|u_j|^p)|u_j|^{p-2}u_j\phi\,dx \to \int_{\R^N}|u_0|^{2p-2}u_0\phi\,dx.
\]
\end{enumerate}
\end{prop}
\begin{proof}
For a proof of this proposition, we refer to \cite{S}.
\end{proof}

\begin{prop}\label{convergenceN}
Fix $2 < p < \frac{2N}{N-2}$. Let $\{\alpha_j\} > 0$ be a sequence converging to $N$ and
$\{u_j\}\subset H_r^1(\R^N)$ be a sequence converging weakly in $H^1(\R^N)$ to some $u_0 \in H^1_r(\R^N)$. 
Then as $j \to \infty$ the following holds:
\begin{enumerate}[$(i)$]
\item 
$\displaystyle
\int_{\R^N}\left(\frac{1}{|\cdot|^{N-\alpha_j}}*|v_j|^p\right)|v_j|^p\,dx \to \left(\int_{\R^N}|v_0|^p\,dx\right)^2;$
\bigskip

\item for any $\phi \in H^1(\R^N)$,
\[
\int_{\R^N}\left(\frac{1}{|\cdot|^{N-\alpha_j}}*|v_j|^p\right)|v_j|^{p-2}v_j\phi\,dx \to \int_{\R^N}|v_0|^p\,dx\int_{\R^N}|v_0|^{p-2}v_0\phi\,dx.
\]
\end{enumerate}
\end{prop}
\begin{proof}
For {$(i)$}, we decompose as
\[
\begin{aligned}
&\int_{\R^N}\left(\frac{1}{|\cdot|^{N-\alpha_j}}*|v_j|^p\right)|v_j|^p\,dx  \\
&\quad=\int_{\R^N}\left(\frac{1}{|\cdot|^{N-\alpha_j}}*|v_j|^p\right)(|v_j|^p-|v_0|^p)\,dx
+\int_{\R^N}\left(\frac{1}{|\cdot|^{N-\alpha_j}}*(|v_j|^p-|v_0|^p)\right)|v_0|^p\,dx \\
&\quad\quad+\int_{\R^N}\left(\frac{1}{|\cdot|^{N-\alpha_j}}*|v_0|^p\right)|v_0|^p\,dx \\
&\quad=: A_j+B_j+C_j.
\end{aligned}
\]
Observe from Proposition \ref{est-identity} that
\[
|A_j| \leq C\|v_j\|_{H^1(\R^N)}^p \left\|~ |v_j|^p-|v_0|^p ~\right\|_{L^1(\R^N)} 
\]
which goes to $0$ as $j \to \infty$ by the compact Sobolev embedding $H^1_r(\R^N) \hookrightarrow L^p(\R^N)$. 
The same argument in the proof of $(ii)$ of Proposition \ref{est-identity} also applies to show that there exists a constant $C > 0$ independent of $\alpha_j$ such that
\[
|B_j| \leq C\left(\left\| ~|v_j|^p-|v_0|^p ~\right\|_{L^{2^*/p}(\R^N)} +\left\| ~|v_j|^p-|v_0|^p ~\right\|_{L^1(\R^N)}\right)\|v_0\|_{L^p(\R^N)}^p,
\]
which also goes to $0$ as $j\to\infty$. 
Finally $C_j$ goes to  $\left(\int_{\R^N}|v_0|^p\,dx\right)^2$ as $j\to\infty$ by $(i)$.

The idea of proof of {$(i)$} is equally applicable to prove {$(ii)$}. We omit it. 
\end{proof}

\section{Limit profile of ground states as $\alpha \to 0$}
In this section, we prove Theorem \ref{profile0}.
We choose an arbitrary $p \in (1,\, N/(N-2))$ and fix it throughout this section.
We denote the ground state energy level of $J_\alpha$ by $E_\alpha$. In other word, $E_\alpha = J_\alpha(u_\alpha)$ 
where $u_\alpha$ is a ground state solution to \eqref{main-eq}.
The ground state energy level $E_\alpha$ of  \eqref{main-eq} satisfies the mountain pass characterization, i.e.,
\[
E_\alpha := \min_{u  \in H^1(\R^N) \setminus \{0\}}\max_{t\geq 0}J_\alpha(tu).
\]
Recall that 
\[
J_0(u) = \frac12\int_{\R^N}|\nabla u|^2+u^2 \,dx -\frac{1}{2p}\int_{\R^N}|u|^{2p}\,dx \quad \text{on } H^1(\R^N),
\]
whose Euler-Lagrange equation is \eqref{limit-eq0}.
We define the mountain pass level of $J_0$ by
\[
E_0 := \min_{u  \in H^1(\R^N) \setminus \{0\}}\max_{t\geq 0}J_0(tu).
\]
It is a well known fact that $E_0$ is the ground state energy level of $J_0$. Namely, 
\[
E_0 = \min\{J_0(u) ~|~ u \in H^1(\R^N),\, J_0'(u) = 0,\, u \not\equiv 0 \}.
\]
{ The following lemma is proved in Claim 1 of Proposition 4.1 in \cite{RV}.
\begin{lemma}\label{upper-ener-est0}
There holds
\[
\lim_{\alpha \to 0} E_{\alpha} = E_0.
\] 
\end{lemma}}

Choose any sequences $\{\alpha_j\} > 0$ converging to $0$ and $\{u_{\alpha_j}\}$ of positive radial ground states to \eqref{main-eq}.
\begin{lemma}\label{conv0}
There exists a positive radial solution $u_0 \in H^1(\R^N)$ to \eqref{limit-eq0} such that
$\{u_{\alpha_j}\}$ converges to $u_0$ in $H^1(\R^N)$, up to a subsequence.
\end{lemma}
\begin{proof}
Multiplying the equation \eqref{main-eq} by $u_{\alpha_j}$ and integrating by part, we get
\begin{equation}\label{equ1}
\left(\frac{1}{2}-\frac{1}{2p}\right)\|u_{\alpha_j}\|^2_{H^1(\R^N)} = J_{\alpha_j}(u_{\alpha_j})
\end{equation}
so $\|u_{\alpha_j}\|_{H^1(\R^N)}$ is uniformly bounded for $j$ by Lemma \ref{upper-ener-est0}.
Then up to a subsequence, $\{u_{\alpha_j}\}$ weakly converges in $H^1(\R^N)$ to some nonnegative radial function $u_0 \in H^1(\R^N)$.  
From Proposition \ref{convergence0} and the weak convergence of $\{u_{\alpha_j}\}$, one is able to deduce $u_0$ is a weak solution of \eqref{limit-eq0}.
In addition, we again multiply the equation \eqref{main-eq} by $u_{\alpha_j}$, multiply the equation \eqref{limit-eq0} by $u_0$ and using Proposition \ref{convergence0} to get
\[
\|u_{\alpha_j}\|^2_{H^1(\R^N)} = \int_{\R^N}(I_{\alpha_j}*|u_{\alpha_j}|^p)|u_{\alpha_j}|^p\,dx \to \int_{\R^N}|u_0|^{2p}\,dx = \|u_0\|^2_{H^1(\R^N)}
\quad\text { as } j \to \infty.
\]
Combining this with the weak convergence of $\{u_{\alpha_j}\}$, we obtain the strong convergence of $\{u_{\alpha_j}\}$ to $u_0$ in $H^1(\R^N)$.

Now, it remains to prove $u_0$ is positive. Observe from Corollary \ref{est-riesz-uniform} and Sobolev inequality that
\begin{equation}\label{ineq1}
\begin{aligned}
\|u_{\alpha_j}\|^2_{H^1(\R^N)} &= \int_{\R^N}(I_{\alpha_j}*|u_{\alpha_j}|^p)|u_{\alpha_j}|^p\,dx 
\leq C\left\|u_{\alpha_j}\right\|^{2p}_{L^{\frac{2Np}{N+{\alpha_j}}}(\R^N)} \\
&\leq C\|u_{\alpha_j}\|^{2p}_{H^1(\R^N)}. 
\end{aligned}
\end{equation}
Here, $C$ is a universal constant independent of $j$. 
%and  $D(N,\alpha_j)$ is given by
%\[
%D(N,\alpha_j) := \frac{\Gamma(\frac{N-\alpha_j}{2})}{\Gamma(\frac{\alpha_j}{2})\pi^{N/2}2^{\alpha_j}}C\left(N,\alpha_j,\frac{2N}{N+\alpha_j}\right)
%\]
%where $C\left(N,\alpha,\frac{2N}{N+\alpha}\right) = \pi^{\frac{N-\alpha}{2}}\frac{\Gamma(\alpha/2)}{\Gamma((N+\alpha)/2)}\left(\frac{\Gamma(N)}{\Gamma(N/2)}\right)^{\alpha/N}$ is the sharp constant in Proposition \ref{HLS}.
Then, dividing both sides of \eqref{ineq1} by $\|u_{\alpha_j}\|^2_{H^1(\R^N)}$ and passing to a limit, 
we obtain a uniform lower bound for $\|u_{\alpha_j}\|_{H^1(\R^N)}$ which implies $u_0$ is nontrivial due to the strong convergence of $\{u_{\alpha_j}\}$.
Since $u_0$ is nonnegative, it is positive from the maximum principle. This completes the proof. 
\end{proof}

%\begin{lemma}\label{lower-ener-est0}
%{ There holds }
%\[
%\liminf_{\j\to\infty} E_{\alpha_j} \geq E_0
%\]
%\end{lemma}
%\begin{proof}
%We see from the mountain pass characterization of ground states to \eqref{main-eq}, Proposition \ref{convergence0} and Lemma \ref{conv0} that for every $t > 0$,
%\[
%\begin{aligned}
%E_{\alpha_j} &\geq J_{\alpha_j}(tu_{\alpha_j}) \\
%&= J_0(tu_{\alpha_j}) +t^{2p}\left(\frac{1}{2p}\int_{\R^N}|u_{\alpha_j}|^{2p}\,dx -\frac{1}{2p}\int_{\R^N}(I_{\alpha_j}*|u_{\alpha_j}|^p)|u_{\alpha_j}|^p\,dx\right) \\
%&= J_0(tu_{\alpha_j}) +o(1)t^{2p} 
%\end{aligned}
%\]
%as $j \to \infty$.
%We take $t = t_{\alpha_j}$ such that $t_{\alpha_j}$ satisfies
%\[
%J_0(t_{\alpha_j} u_{\alpha_j}) = \max_{t\geq0}J_0(tu_{\alpha_j}).
%\]
%Elementary computation shows
%\[
%t_{\alpha_j} = \left(\frac{\|u_{\alpha_j}\|^2_{H^1(\R^N)}}{\int_{\R^N}|u_{\alpha_j}|^{2p}\,dx}\right)^{\frac{1}{2p-2}}
%\to \quad \left(\frac{\|u_0\|^2_{H^1(\R^N)}}{\int_{\R^N}|u_0|^{2p}\,dx}\right)^{\frac{1}{2p-2}} = 1 \quad \text{as } j \to \infty
%\]
%as in the proof of Lemma \ref{upper-ener-est0}. Then we get
%\[
%E_{\alpha_j} \geq J_0(t_{\alpha_j}u_{\alpha_j}) +o(1) \geq E_0 +o(1), 
%\]
%where the last inequality comes from the mountain pass characterization of $E_0$. The proof is complete. 
%\end{proof}
%
%Combining Lemma \ref{upper-ener-est0} with Lemma \ref{lower-ener-est0}, we conclude
%\begin{equation}\label{equ2}
%\lim_{j\to\infty} E_{\alpha_j} = E_0.
%\end{equation}
Then the following lemma follows. 
\begin{lemma}\label{energy}
There holds 
\[
J_0(u_0) = E_0.
\]
In other words, $u_0$ is a unique positive radial ground state to \eqref{limit-eq0}.
\end{lemma}
\begin{proof}
We see from Proposition \ref{convergence0}, Lemma \ref{upper-ener-est0} and Lemma \ref{conv0} that
\[
\begin{aligned}
E_0 & = \lim_{j\to\infty}E_{\alpha_j} = \lim_{j\to\infty}J_{\alpha_j}(u_{\alpha_j}) \\
&= \lim_{j\to\infty} \left(\frac{1}{2}\|u_{\alpha_j}\|^2_{H^1(\R^N)} -\frac{1}{2p}\int_{\R^N}(I_{\alpha_j}*|u_{\alpha_j}|^p)|u_{\alpha_j}|^p\,dx\right) \\
& = \frac{1}{2}\|u_0\|^2_{H^1(\R^N)} -\frac{1}{2p}\int_{\R^N}|u_0|^{2p}\,dx = J_0(u_0),
\end{aligned}
\]
which proves the lemma. 
\end{proof}

Now, we are ready to complete the proof of Theorem \ref{profile0}. 
Let $\{u_\alpha\} \subset H^1_r(\R^N)$ be a family of positive radial ground states to \eqref{main-eq} for $\alpha$ near $0$. 
Suppose $\{u_\alpha\}$ does not converge in $H^1(\R^N)$ to the unique positive radial ground state $u_0$ of \eqref{limit-eq0}.
Then there exists a positive number $\ep_0$ and a sequence $\{\alpha_j\} \to 0$ such that $\|u_{\alpha_j}-u_0\|_{H^1(\R^N)} \geq \ep_0$
which contradicts to Lemma \ref{conv0} and Lemma \ref{energy}.

\section{Limit profile of ground states as $\alpha \to N$}
In this section, we prove Theorem \ref{prop-limit-eqN} and \ref{profileN}. Choose and fix an arbitrary $p \in (2,\, 2N/(N-2))$.
By deleting the coefficient of Riesz potential term from \eqref{main-eq}, we obtain the equation 
\begin{equation}\label{rescaled-eq}
\left\{\begin{array}{rcl}
-\Delta v +v & = &(1/|\cdot|^{N-\alpha}*|v|^p)|v|^{p-2}v  \quad \text{in } \R^N, \\
\lim_{x \to \infty}v(x) & = &0.
\end{array}\right.
\end{equation}
For simplicity, we still denote  by $J_\alpha$ the energy functional of \eqref{rescaled-eq}.
It is clear that the ground state energy level $E_\alpha$ of  \eqref{rescaled-eq} also satisfies the mountain pass characterization, 
\[
E_\alpha := \min_{v \in H^1(\R^N) \setminus \{0\}}\max_{t\geq 0}J_\alpha(tv).
\]
Recall that as $\alpha \to N$, $J_\alpha$ approaches to a limit functional
\[
J_N(v) = \frac12\int_{\R^N}|\nabla v|^2+v^2 \,dx -\frac{1}{2p}\left(\int_{\R^N}|v|^p\,dx\right)^2 \quad \text{on } H^1(\R^N),
\]
whose Euler-Lagrange equation is \eqref{limit-eqN}.

\subsection{Proof of Theorem \ref{prop-limit-eqN}}
We first prove Theorem \ref{prop-limit-eqN}.
To prove {$(i)$}, we let $v_1$ and $v_2$ be two positive radial ground states to \eqref{limit-eqN}. 
By defining $a_1 = \int_{\R^N}|v_1|^p\,dx$ and $a_2 = \int_{\R^N}|v_2|^p\,dx$, they are positive radial solutions of the equations
$-\Delta w +w  = a_1|w|^{p-2}w$ and $-\Delta w +w  = a_2|w|^{p-2}w$ respectively.
We note that $(a_1/a_2)^{1/(p-2)}v_1$ satisfies the latter equation. 
The classical result due to Kwong \cite{K} says that a positive radial solution of the latter(also former) equation is unique so one must have $(a_1/a_2)^{1/(p-2)}v_1 \equiv v_2$.
Since both of $v_1$ and $v_2$ satisfy the equation \eqref{limit-eqN}, we can conclude $(a_1/a_2)^{1/(p-2)} = 1$. 

{We next prove $(ii)$}. Let $v_0$ be the positive and radial ground state of \eqref{limit-eqN}. 
Let $a_0 = \int_{\R^N}v_0^p\,dx$. As discussed above, $v_0$ is a unique positive radial solution of
\begin{equation}\label{equ4}
-\Delta w +w = a_0|w|^{p-2}w.
\end{equation}
It is a well known fact that the linearized operator of \eqref{equ4} at $v_0$, given by
\[
L(\phi) := -\Delta \phi +\phi -(p-1)a_0v_0^{p-2}\phi
\]
admits only { solutions of the form
\begin{equation}\label{kernel-form}
\phi = \sum_{i=1}^Nc_i\partial_{x_i}v_0, \quad c_i \in \R
\end{equation}
in the space $L^2(\R^N)$. 
To the contrary, suppose that \eqref{linearized-limit-eqN} has a nontrivial solution $\phi \in L^2(\R^N)$, which is not of the form \eqref{kernel-form}.
Then we may assume that $\phi$ is $L^2$ orthogonal to $\partial_{x_i}v_0$ for every $i = 1,\dots,N$.
By denoting $\lambda := p\int_{\R^N}v_0^{p-1}\phi\,dx$, we see $L(\phi) = \lambda v_0^{p-1}$ so $\lambda$ should not be $0$. Observe that
\[
L\left(\frac{\lambda}{(2-p)a_0}v_0\right) = \frac{\lambda}{(2-p)a_0}L\left(v_0\right)
= \frac{\lambda}{(2-p)a_0}(-\Delta v_0 +v_0 -(p-1)a_0v_0^{p-1}) = \lambda v_0^{p-1}.
\]
This shows $L(\phi-\frac{\lambda}{(2-p)a_0}v_0) \equiv 0$, which implies that there are some $c_i \in \R$ such that
\begin{equation}\label{difference}
\phi -\frac{\lambda}{(2-p)a_0}v_0 = \sum_{i=1}^Nc_i\partial_{x_i}v_0. 
\end{equation}
We claim that $c_i = 0$ for all $i$.  Indeed, by multiplying the (LHS) of \eqref{difference} by $\partial_{x_j}v_0$ and integrating, we get
\[
\int_{\R^N}\phi \partial_{x_j}v_0\,dx-\frac{\lambda}{(2-p)a_0}\int_{\R^N}v_0\partial_{x_j}v_0\,dx
= -\frac{\lambda}{(2-p)a_0}\int_{\R^N}\frac12\partial_{x_j}(v_0^2)\,dx = 0.
\]
On the other hand, by multiplying \eqref{difference} by $\partial_{x_j}v_0$ and integrating, we get
\begin{multline*}
c_j\int_{\R^N}(\partial_{x_j}v_0)^2\,dx +\sum_{i\neq j} c_i\int_{\R^N}\partial_{x_i}v_0\partial_{x_j}v_0\,dx \\
= c_j\int_{\R^N}(\partial_{x_j}v_0)^2\,dx +\sum_{i\neq j} c_i\int_{\R^N}\frac{x_ix_j}{r^2}v_0'(r)\,dx
=c_j\int_{\R^N}(\partial_{x_j}v_0)^2\,dx
\end{multline*}
since $\frac{x_ix_j}{r^2}v_0'(r)$ is odd in variables $x_i$ and $x_j$.
Combining these two integrals, the claim follows. 

Now, observe that
\[
\lambda = p\int_{\R^N}v_0^{p-1}\phi\,dx = p\frac{\lambda}{(2-p)a_0}\int_{\R^N}v_0^p\,dx =p\frac{\lambda}{2-p}.
\]
This says $p = 1$ which contradicts with the hypothesis for $p$.
This completes a whole proof of Theorem \ref{prop-limit-eqN}. } \\

\subsection{Proof of Theorem \ref{profileN}}
It remains to prove Theorem \ref{profileN}.
Choose an arbitrary positive sequence $\{\alpha_j\} \to N$ and an arbitrary sequence $\{v_{\alpha_j}\}$ of positive radial ground states to \eqref{rescaled-eq}.
Analogously arguing to the previous section, one can see the following proposition also holds true.  
\begin{prop}
Choosing a subsequence, $\{v_{\alpha_j}\}$ converges in $H^1(\R^N)$ to the unique positive radial ground state of \eqref{limit-eqN}. 
\end{prop}
\begin{proof}
We follow the same line in Section $3$. It is proved in Claim 1 of Proposition 5.1 in \cite{RV} that 
\[
\lim_{\alpha\to N}E_\alpha = E_N,
\]
where $E_N := \min_{v \in H^1(\R^N) \setminus \{0\}}\max_{t\geq 0}J_N(tv)$.
This implies $\|v_{\alpha_j}\|_{H^1}$ is bounded and consequently, has a weak subsequential limit $v_0 \in H^1_r(\R^N)$, which is radial and nonnegative.
Proposition \ref{convergenceN} says $v_0$ is a solution of \eqref{limit-eqN}. 
Again using Proposition \ref{convergenceN}, we have 
\[
\|v_{\alpha_j}\|_{H^1}^2 = \int_{\R^N}(\frac{1}{|\cdot|^{N-\alpha_j}}*v_{\alpha_j}^p)v_{\alpha_j}^{p}\,dx = \left(\int_{\R^N}v_0^p\,dx\right)^2 +o(1) = \|v_0\|_{H^1}^2+o(1),
\]
which says $H^1$ strong convergence of $\{v_{\alpha_j}\}$. We now invoke $(ii)$ of Proposition \ref{est-identity} to see
\[
\|v_{\alpha_j}\|_{H^1}^2 = \int_{\R^N}(\frac{1}{|\cdot|^{N-\alpha_j}}*v_{\alpha_j}^p)v_{\alpha_j}^{p}\,dx \leq C\|v_{\alpha_j}\|_{H^1}^p\|v_{\alpha_j}\|_{L^p}^p
\leq C\|v_{\alpha_j}\|_{H^1}^{2p},
\]
where $C$ is independent of $j$. This shows $v_0$ is nontrivial so that it is positive by the strong maximum principle. 
Finally, as in Lemma \ref{energy}, we can check $J_N(v_0) = E_N$, which completes the proof. 
\end{proof}

Now, we shall complete the proof of Theorem \ref{profileN}.
Fix $p \in (2,\, 2N/(N-2))$. Let $\{u_\alpha\}$ be a family of positive radial ground states to \eqref{main-eq} for $\alpha$ close to $N$. 
Then it is clear that the rescaled functions $v_\alpha := s(N,\alpha) u_\alpha$ constitute a family of positive radial ground states of \eqref{rescaled-eq} by a direct computation. 
Therefore as in the proof of Theorem \ref{profile0}, one may conclude $\lim_{\alpha\to N}\|v_\alpha -v_0\|_{H^1(\R^N)} = 0$, where we denote by $v_0$ a unique positive radial solution to \eqref{limit-eqN}.

\section{Uniqueness of ground states}
We begin this section with a simple elliptic estimate. 
\begin{lemma}\label{elliptic-regularity}
Let $\frac{2N}{N+2} \leq q \leq 2$.
Then the operator $(-\Delta+I)^{-1}$ is a bounded from $L^q(\R^N)$ into $H^1(\R^N)$.
\end{lemma}
\begin{proof}
We multiply the equation $-\Delta u + u = f$ by $u$, integrate by part and apply H\"older inequality
\[
\|u\|_{H^1(\R^N)}^2 = \int_{\R^N}fu\,dx \leq \|f\|_{L^q(\R^N)}\|u\|_{L^{q/(q-1)}(\R^N)}. 
\] 
Since $2 \leq q/(q-1) \leq 2N/(N-2)$, Sobolev inequality applies to see
\[
\|u\|_{H^1(\R^N)} \leq C\|f\|_{L^q(\R^N)}.
\]
for some $C$ depending only on $q$ and $N$. Then the density arguments complete the proof.  
\end{proof}

For $p \in (1,\, N/(N-2))$, choose and fix $\alpha_0 \in (0,\, \frac{N(p-1)}{2})$ and define an operator $A(\alpha, u)$ by
\[
A(\alpha,u) := 
\left\{\begin{array}{ll}
u-(-\Delta+I)^{-1}[(I_\alpha*|u|^p)|u|^{p-2}u] &\quad \text{if } \alpha \in (0,\, \alpha_0); \\
u-(-\Delta+I)^{-1}[|u|^{2p-2}u] &\quad\text{if } \alpha = 0. \\
\end{array}\right.
\]
For $p \in (2,\, 2N/(N-2))$, choose and fix $\alpha_N \in (\frac{(N-2)p}{2},\, N)$ and define an operator $B(\alpha, v)$ by
\[
B(\alpha,v) := 
\left\{\begin{array}{ll}
v-(-\Delta+I)^{-1}[(1/|\cdot|^{N-\alpha}*|v|^p)|v|^{p-2}v] &\quad\text{if } \alpha \in (\alpha_N,\, N); \\
v-(-\Delta+I)^{-1}\left[\left(\int_{\R^N}|v|^p\,dx\right)|v|^{p-2}v\right] &\quad\text{if } \alpha = N. \\
\end{array}\right.
\]

\begin{lemma}\label{dble}
The operator $A$ is a continuous map from $[0,\, \alpha_0) \times H^1_r(\R^N)$ into $H^1_r(\R^N)$ and
is continuously differentiable with respect to $u$ on $[0,\, \alpha_0) \times H^1_r(\R^N)$. 
The same conclusion holds true for $B$ which is a map from $(\alpha_N,\, N] \times H^1_r(\R^N)$ into $H^1(\R^N)$.
\end{lemma}
\begin{proof}
We first prove the continuity of $A$.
Let $\{(\alpha_j, u_j)\}$ be a sequence in $ \in [0,\alpha_0) \times H^1_r(\R^N)$ converging to some $(\alpha, u) \in [0, \alpha_0) \times H^1_r(\R^N)$. 
We only deal with the case $\alpha_j \neq 0$ and $\alpha = 0$. Then the remaining cases can be similarly as well as more easily dealt with. 
Lemma \ref{elliptic-regularity} shows that it is sufficient to prove  $(I_{\alpha_j}*|u_j|^p)|u_j|^{p-2}u_j$ converges to $|u|^{2p-2}u$ in $L^q(\R^N)$ for some $q \in [2N/(N+2),\, 2]$.
We select $q = 2p/(2p-1)$.  Since $p \in (1, N/(N-2))$, one can easily see $q$ belongs to the above range.  Then
\begin{equation}\label{equ55}
\begin{aligned}
&\|~(I_{\alpha_j}*|u_j|^p)|u_j|^{p-2}u_j -|u|^{2p-2}u~\|_{L^{2p/(2p-1)}(\R^N)} \\
&\leq \|~((I_{\alpha_j}*|u_j|^p)-|u_j|^p)|u_j|^{p-2}u_j~\|_{L^{2p/(2p-1)}(\R^N)}  + \|~|u_j|^{2p-2}u_j-|u|^{2p-2}u~\|_{L^{2p/(2p-1)}(\R^N)} \\
&\leq \|~(I_{\alpha_j}*|u_j|^p)-|u_j|^p~\|_{L^2(\R^N)}\|~|u_j|^{p-1}~\|_{L^{2p/(p-1)}(\R^N)} +o(1) \\
&\leq C\|~(I_{\alpha_j}*|u|^p)-|u|^p +I_{\alpha_j}*(|u_j|^p-|u|^p)+|u|^p-|u_j|^p~\|_{L^2(\R^N)} +o(1) \\
&\leq C\|~I_{\alpha_j}*(|u_j|^p-|u|^p)~\| +o(1) \leq C\|~|u_j|^p-|u|^p~\|_{L^{2N/(N+2\alpha)}(\R^N)} +o(1) =o(1),
\end{aligned}
\end{equation}
where we used H\"older inequality, Sobolev inequality, the sharp constant estimate in Proposition \ref{est-riesz}.

Differentiating $A$ with respect to $u$, we get
\[
\begin{aligned}
&\frac{\partial A}{\partial u}(\alpha, u)[\phi]  = \\ 
&\left\{\begin{array}{ll}
\phi-(-\Delta+I)^{-1}[p(I_\alpha*|u|^{p-2}u\phi)|u|^{p-2}u  +(p-1)(I_\alpha*|u|^p)|u|^{p-2}\phi ] & \\
&\text{if } \alpha \in (0,\, \alpha_0); \\
\phi-(-\Delta+I)^{-1}[(2p-1)|u|^{2p-2}\phi] & \\
&\text{if } \alpha = 0. \\
\end{array}\right.
\end{aligned}
\]
Then one can apply essentially the same argument to \eqref{equ55} to see $\partial A/\partial u$ is continuous on $[0,\, \alpha_0] \times H^1_r(\R^N)$.

We next address the operator $B$.
Let $\{(\alpha_j, v_j)\}$ be a sequence in $ \in (\alpha_N,\, N] \times H^1_r(\R^N)$ converging to some $(\alpha, v) \in (\alpha_N,\, N] \times H^1_r(\R^N)$. 
We only deal with the case $\alpha = N$ and $\alpha_j \neq N$. 
As the above, it is sufficient to show $(1/|\cdot|^{N-\alpha_j}*|v_j|^p)|v_j|^{p-2}v_j$ converges to $\left(\int_{\R^N}|v|^p\,dx\right)|v|^{p-2}v$ in $L^{p/(p-1)}(\R^N)$ for the continuity of $B$.
This follows by arguing similarly to \eqref{equ55} with Proposition \ref{est-identity}.
\end{proof}

\begin{lemma}\label{local-unique}
Let $u_0$ be a unique positive radial ground state of \eqref{limit-eq0}.
Then, there exists a neighborhood $U_0 \subset [0,\, \alpha_0) \times H^1_r(\R^N)$ of a point $(0, u_0) \in [0,\, \alpha_0) \times H^1_r(\R^N)$ such that the equation \eqref{main-eq} admits
a unique solution in $U_0$.
Let $v_0$ be a unique positive radial ground state of \eqref{limit-eqN}.
Then, there exists a neighborhood $U_N \subset (\alpha_N,\, N] \times H^1_r(\R^N)$ of a point $(N, v_0) \in (\alpha_N,\, N] \times H^1_r(\R^N)$ such that the equation \eqref{rescaled-eq} admits
a unique solution in $U_N$.
\end{lemma}
\begin{proof}
We only prove the former assertion. The latter assertion follows similarly. 
We claim that the linearized operator of $A$ with respect to $u$ at $(0, u_0)$, namely $\frac{\partial A}{\partial u}(0,u_0)$ is a linear isomorphism from $H^1_r(\R^N)$ into $H^1_r(\R^N)$.
Observe
\[
\frac{\partial A}{\partial u}(0,u_0)[\phi] = \phi-(2p-1)(-\Delta+I)^{-1}[u_0^{2p-2}\phi].
\]
Since $u_0$ decays exponentially, the map $\phi \mapsto u_0^{2p-2}\phi$ is compact from $H^1_r(\R^N)$ into $L^2(\R^N)$ 
so the composite map $\phi \mapsto (-\Delta+I)^{-1}[u_0^{2p-2}\phi]$ is also compact from $H^1_r(\R^N)$ into $H^1_r(\R^N)$. This also shows $\frac{\partial A}{\partial u}(0,u_0)$ is bounded. 
One can deduce from the radial linearized nondegeneracy of $u_0$ that the kernel of $\frac{\partial A}{\partial u}(0,u_0)$ is trivial. 
Then the Fredholm alternative applies to see that $\frac{\partial A}{\partial u}(0,u_0)$ is onto map so the claim is proved. 
We invoke the implicit function theorem to complete the proof. 

\end{proof}

Now, we claim that \eqref{main-eq} admits a unique positive radial ground state for $p \in (1,\, N/(N-2))$ and $\alpha$ close to $0$.
Suppose the contrary. Then there exists sequences $\{\alpha_j\} > 0$, $\{u_{\alpha_j}^1\} \subset H^1_r(\R^N)$ and $\{u_{\alpha_j}^2\} \subset H^1_r(\R^N)$ such that
$\alpha_j \to 0$ as $j \to \infty$, $\{u_{\alpha_j}^1\}$ and $\{u_{\alpha_j}^2\}$ are sequences of positive radial ground states of \eqref{main-eq} and $u_{\alpha_j}^1 \neq u_{\alpha_j}^2$ for all $j$. 
Theorem \ref{profile0} tells us that  both of $\{u_{\alpha_j}^1\}$ and $\{u_{\alpha_j}^2\}$ converge to a unique positive radial solution $u_0$ of \eqref{limit-eq0} in $H^1(\R^N)$. 
This however contradicts with Lemma \ref{local-unique} and thus shows the uniqueness of a positive radial ground state of \eqref{main-eq} for $p \in (1,\, N/(N-2))$ and $\alpha$ close to $0$.
Note that the analogous conclusion holds for a family of ground states $\{v_\alpha\}$ of \eqref{rescaled-eq} when $p \in (2,\, 2N/(N-2))$ and $\alpha$ close to $N$. 
Scaling back, this also shows the uniqueness of a positive radial ground state of \eqref{main-eq} when $p \in (2,\, 2N/(N-2))$ and $\alpha$ close to $N$.

\section{Nondegeneracy of ground states}
\subsection{Nondegeneracy for $\alpha$ near 0}
Throughout this subsection, we fix $N = 3$ due to the restriction $p \in [2, N/(N-2))$.
We begin with proving a convergence lemma similar to Proposition \ref{convergence0} but slightly different.
\begin{lemma}\label{conv-linearized0}
For given $p \in [2, 3)$, let $u_\alpha$ be a family of the unique positive radial ground states of \eqref{main-eq} and $u_0$ be the positive radial ground state to \eqref{limit-eq0}.
Then, for any $\{\alpha_j\} \to 0$ and $\{\psi_j\}, \{\phi_j\} \subset H^1$ weakly $H^1$ converging to $\phi_0$ and $\psi_0$, there holds the following: 
\begin{enumerate}[$(i)$]
\item
$\displaystyle\int_{\R^N}(I_{\alpha_j}*(u_{\alpha_j}^{p-1}\phi_j))u_{\alpha_j}^{p-1}\psi_j\,dx \to \int_{\R^N}u_0^{2p-2}\phi_0\psi_0\,dx$;
 
\item
$\displaystyle\int_{\R^N}(I_{\alpha_j}*u_{\alpha_j}^p)u_{\alpha_j}^{p-2}\phi_j\psi_j\,dx \to \int_{\R^N}u_0^{2p-2}\phi_0\psi_0\,dx.$
\end{enumerate}
\end{lemma}
\begin{proof}

We first note that $u_0^{p-1}\phi_j$ is compact in $L^2$ due to the uniform decaying property of $u_0$. 
Then one has from H\"older inequality, 
\[
\begin{aligned}
\|u_{\alpha_j}^{p-1}\phi_j-u_0^{p-1}\phi_0\|_{L^2} &\leq \|(u_{\alpha_j}^{p-1}-u_0^{p-1})\phi_j\|_{L^2} + \|u_0^{p-1}\phi_j-u_0^{p-1}\phi_0\|_{L^2} \\
&\leq \|u_{\alpha_j}^{p-1}-u_0^{p-1}\|_{L^{2p/(p-1)}}\|\phi_j\|_{L^{2p}} +o(1) \\
&\leq C\| ~|u_{\alpha_j}-u_0|(|u_{\alpha_j}|^{p-2}+|u_0|^{p-2})~\|_{L^{2p/(p-1)}}\|\phi_j\|_{L^{2p}} +o(1)\\
&\leq \|u_{\alpha_j}-u_0\|_{L^{2p}}(\|u_{\alpha_j}\|_{L^{2p}}^{p-2}+\|u_0\|_{L^{2p}}^{p-2} )\|\phi_j\|_{L^{2p}}+o(1)\\
&= o(1),
\end{aligned}
\]
from which we deduce $u_{\alpha_j}^{p-1}\phi_j$ is also compact in $L^2$.
We decompose the LHS of $(i)$ by 
\[
\begin{aligned}
&\int_{\R^3}(I_{\alpha_j}*(u_{\alpha_j}^{p-1}\phi_j))u_{\alpha_j}^{p-1}\psi_j\,dx  \\
& =\int_{\R^3}u_{\alpha_j}^{p-1}\phi_ju_{\alpha_j}^{p-1}\psi_j\,dx + \int_{\R^3}((I_{\alpha_j}*(u_{\alpha_j}^{p-1}\phi_j))-u_{\alpha_j}^{p-1}\phi_j)u_{\alpha_j}^{p-1}\psi_j\,dx \\
&= \int_{\R^3}u_{\alpha_j}^{p-1}\phi_ju_{\alpha_j}^{p-1}\psi_j\,dx + \int_{\R^3}(I_{\alpha_j}*(u_{\alpha_j}^{p-1}\phi_j-u_0^{p-1}\phi_0))u_{\alpha_j}^{p-1}\psi_j\,dx  \\  
&\qquad\qquad+\int_{\R^N}(I_{\alpha_j}*(u_0^{p-1}\phi_0)-u_{\alpha_j}^{p-1}\phi_j)u_{\alpha_j}^{p-1}\psi_j\,dx \\
&=\int_{\R^3}u_0^{2p-2}\phi_0\psi_0\,dx +\int_{\R^3}(I_{\alpha_j}*(u_{\alpha_j}^{p-1}\phi_j-u_0^{p-1}\phi_0))u_{\alpha_j}^{p-1}\psi_j\,dx +o(1),
\end{aligned}
\] 
where we used H\"older inequality, Proposition \ref{est-identity} and $L^2$ compactness of both of $\{u_{\alpha_j}^{p-1}\phi_j\}$ and $\{u_{\alpha_j}^{p-1}\psi_j\}$.  
We now estimate by using Corollary \ref{est-riesz-uniform} that
\[
\begin{aligned}
&\int_{\R^N}(I_{\alpha_j}*(u_{\alpha_j}^{p-1}\phi_j-u_0^{p-1}\phi_0))u_{\alpha_j}^{p-1}\psi_j\,dx \\
&\leq \|I_{\alpha_j}*(u_{\alpha_j}^{p-1}\phi_j-u_0^{p-1}\phi_0)\|_{L^{2N/(N-2\alpha_j)}}\|u_{\alpha_j}^{p-1}\psi_j\|_{L^{2N/(N+2\alpha_j)}} \\
&\leq C\|u_{\alpha_j}^{p-1}\phi_j-u_0^{p-1}\phi_0\|_{L^2}\|u_{\alpha_j}^{p-1}\psi_j\|_{L^1}^{\frac{2\alpha_j}{N}}\|u_{\alpha_j}^{p-1}\psi_j\|_{L^2}^{1-\frac{2\alpha_j}{N}} = o(1).
\end{aligned}
\]
This proves the assertion $(i)$. The proof of $(ii)$ follows exactly the same line. 
\end{proof}

%Next, we prove the uniform $L^\infty$ bound of $v_{\alpha}$ for small $\alpha > 0$.
%\begin{lemma}\label{uniform-bound0}
%For given $p \in [2, 3)$, let $u_\alpha$ be a family of the unique positive radial ground states of \eqref{main-eq}.
%Then one has the uniform $L^\infty$ bound of $u_\alpha$ for sufficiently small $\alpha > 0$.
%In particular, $u_\alpha$ has the uniform $L^q$ bound for every $q \in [2, \infty]$ by the interpolation. 
%\end{lemma}
%\begin{proof}
%Define a function
%\[
%c(x) := 1-\left(I_\alpha *u_\alpha^p\right)u_{\alpha}^{p-2}
%\]
%so that the equation $-\Delta u_\alpha + c(x)u_\alpha = 0$ holds on $\R^3$. 
%For any given $x_0 \in \R^3$, the well-known de Giorgi-Nash-Moser estimate says that if there exists $q > 3/2$ such that $c \in L^q(B_1(x_0))$, then 
%\[
%\|v_\alpha\|_{L^\infty(B_1(x_0))} \leq C\|v_\alpha\|_{L^2(B_1(x_0))},
%\] 
%where $C$ depends only on $\|c\|_{L^q}$. We select $1/q = p/6-\alpha/3+(p-2)/6$. Then $q > 3/2$ for $\alpha$ near $0$ since $p < 3$. 
%Using Corollary \ref{est-riesz-uniform}, we see that
%\[
%\|c\|_{L^q} \leq \|I_\alpha*u_\alpha^p\|_{L^{1/(p/6-\alpha/3)}}\|u_\alpha\|_{L^6}^{p-2} \leq C\|u_\alpha^{p}\|_{L^{6/p}}\|u_\alpha\|_{L^6}^{p-2}
%\leq C\|u_\alpha\|_{H^1}^{2p-2},
%\]
%where $C$ is independent of small $\alpha > 0$. This completes the proof. 
%\end{proof}

Now we are ready to prove nondegeneracy of ground states $u_\alpha$ to \eqref{main-eq} near $0$. 
\begin{prop}\label{nondeg0}
For given $p \in [2, 3)$, let $u_\alpha$ be a family of unique positive radial ground states of \eqref{main-eq}. Then for $\alpha > 0$ sufficiently close to $0$, the linearized equation of \eqref{main-eq} at $u_\alpha$, given by
\begin{equation}\label{linearized-eq}
-\Delta \phi +\phi -p(I_\alpha * (u_\alpha^{p-1}\phi))u_\alpha^{p-1}-(p-1)(I_\alpha * u_\alpha^p)u_\alpha^{p-2}\phi  = 0 \quad \text{in } \R^3,
\end{equation}
only admits solutions of the form
\[
\phi = \sum_{i=1}^3 c_i\partial_{x_i}u_\alpha, \qquad c_i \in \R
\]
in the space $L^2(\R^3)$.
\end{prop}
\begin{proof}
Differentiating \eqref{main-eq} with respect to $x_i$, we see that $\partial_{x_i}u_\alpha \in L^2(\R^3)$ solves \eqref{linearized-eq} for all $i =1,\dots,N$.
Define a finite dimensional vector space
\[
V_\alpha := \left\{\sum_{i=1}^3 c_i\partial_{x_i}u_\alpha ~ \bigg|~ c_i \in \R \right\}.
\]
Arguing indirectly, suppose the there exists a sequence $\{\alpha_j\}$ converging to $0$ 
such that for each $j$, there exists a nontrivial solution $\phi_j \in L^2$ of \eqref{linearized-eq}, not belonging to $V_{\alpha_j}$.
We may assume that $\phi_j$ is $L^2$ orthogonal to $V_{\alpha_j}$. 
We claim that any $L^2$ solution $\phi$ of \eqref{linearized-eq} automatically belongs to $H^1(\R^3)$. 
Let us define 
\[
L[\phi] := p(I_\alpha * (u_\alpha^{p-1}\phi))u_\alpha^{p-1}+(p-1)(I_\alpha * u_\alpha^p)u_\alpha^{p-2}\phi.
\]
By elliptic regularity theory, it is enough to show that $L[\phi] \in H^{-1}$. 
It is proved in \cite{MV2} that $u_\alpha \in L^\infty$ so $u_\alpha \in L^q$ for any $2 \leq q \leq \infty$ by interpolation.
Then Proposition \ref{est-riesz} and Proposition \ref{est-identity} imply that for any $\psi \in H^1$,
\[
\begin{aligned}
&|L[\phi]\psi| \leq  \left|\int_{\R^3}p(I_\alpha * (u_\alpha^{p-1}\phi))u_\alpha^{p-1}\psi\,dx\right|  +\left|\int_{\R^3}(p-1)(I_\alpha * u_\alpha^p)u_\alpha^{p-2}\phi\psi\,dx\right| \\
&\leq p\left|\int_{\R^3}(I_\alpha * (u_\alpha^{p-1}\psi))u_\alpha^{p-1}\phi\,dx\right|  +(p-1)\left|\int_{\R^3}(I_\alpha * u_\alpha^p)u_\alpha^{p-2}\phi\psi\,dx\right| \\
&\leq p\|I_\alpha * (u_\alpha^{p-1}\psi)\|_{L^2}\|u_\alpha\|_{L^\infty}^{p-1}\|\phi\|_{L^2}+(p-1)\|I_\alpha * u_\alpha^p\|_{L^3}\|u_\alpha\|_{L^\infty}^{p-2}\|\phi\|_{L^2}\|\psi\|_{L^6} \\
&\leq C(\|u_\alpha\|_{H^1}^{p-1}\|u_\alpha\|_{L^\infty}^{p-1}\|\phi\|_{L^2} + \|u_\alpha\|_{L^{3p/(1+\alpha)}}^p\|u_\alpha\|_{L^\infty}^{p-2}\|\phi\|_{L^2} )\|\psi\|_{H^1},
\end{aligned}
\]
which shows $L[\phi] \in H^{-1}$.
%In fact, we formally multiply both side of the equation \eqref{linearized-eq} by $\phi$ and integrate by part to see
%\[
%\begin{aligned}
%\|\phi\|_{H^1(\R^N)}^2 &= p\frac{\Gamma(\frac{N-\alpha}{2})}{\Gamma(\frac{\alpha}{2})\pi^{N/2}2^\alpha}\int_{\R^N}\!\!\int_{\R^N}\frac{(u_\alpha^{p-1}\phi)(x)(u_\alpha^{p-1}\phi)(y)}{|x-y|^{N-\alpha}}\,dxdy \\
%&\qquad\qquad\qquad\qquad\qquad+(p-1)\int_{\R^N}(I_\alpha * u_\alpha^p)u_\alpha^{p-2}\phi^2 \,dx \\
%&\leq K(N,\alpha,p) \left(\|u_\alpha^{p-1}\phi\|_{L^{2N/(N+\alpha)}(\R^N)}^2 +\|(I_\alpha * u_\alpha^p)u_\alpha^{p-2}\|_{L^\infty(\R^N)}\|\phi\|_{L^2(\R^N)}^2  \right) \\
%&\leq K(N,\alpha,p) \left(\|u_{\alpha}^{p-1}\|_{L^{2N/\alpha}(\R^N)}+\|(I_\alpha * u_\alpha^p)u_\alpha^{p-2}\|_{L^\infty(\R^N)}\right)\|\phi\|_{L^2(\R^N)}^2 \\
%&< \infty,
%\end{aligned}
%\]
%where we used Hardy-Littlewood-Sobolev inequality, H\"older inequality and exponential decaying property of $u_\alpha$. 
%Then the standard density argument shows $\phi$ is really in $H^1(\R^N)$. 
We normalize $\phi_j$ as $\|\phi_j\|_{H^1} = 1$. As $j \to \infty$, 
one is possible to deduce from Lemma \ref{conv-linearized0} that $\phi_j$ weakly converges in $H^1$ to some $\phi_0 \in H^1$ which satisfies
\[
-\Delta \phi_0 +\phi_0 -(2p-1)u_0^{2p-2}\phi_0  = 0,
\] 
where $u_0$ is a unique positive radial solution of \eqref{limit-eqN}. Repeatedly applying Lemma \ref{conv-linearized0}, one also has
\[
\begin{aligned}
1 &= \|\phi_j\|_{H^1}^2 \\
&= p\int_{\R^3}I_{\alpha_j}*(u_{\alpha_j}^{p-1}\phi_j)u_{\alpha_j}^{p-1}\phi_j\,dx +(p-1)\int_{\R^3}(I_{\alpha_j} * u_{\alpha_j}^p)u_{\alpha_j}^{p-2}\phi_j^2 \,dx \\
&=(2p-1) \int_{\R^3}u_0^{2p-2}\phi_0^2\,dx +o(1) \quad \text{as } j \to \infty.	
\end{aligned}
\]
This shows $\phi_0$ is nontrivial. 
Finally we note that for all $i = 1,2,3$,
\[
0 = \int_{\R^3}\partial_{x_i}u_{\alpha_j}\phi_j\,dx \to \int_{\R^3}\partial_{x_i}u_0\phi_0\,dx \quad \text{as } j \to \infty.
\]
This means that $\phi_0$ is not a linear combination of $\{\partial_{x_i} u_0 ~|~ i =1,2,3\}$.
This contradicts with the linearized nondegeneracy of $u_0$ and completes the proof.
\end{proof}

\subsection{Nondegenracy for $\alpha$ near $N$}
Arguing as in the proof of Proposition \ref{nondeg0}, we also obtain the nondegeneracy result of $\alpha$ near $N$. 
We need lemmas analogous to Lemma \ref{conv-linearized0}.

\begin{lemma}\label{conv-linearizedN}
For given $p \in (2, 2N/(N-2))$, let $v_\alpha$ be a family of unique positive radial ground states of \eqref{rescaled-eq} and $v_0$ be the positive radial ground state to \eqref{limit-eqN}.
Then, for any $\{\alpha_j\} \to 0$ and $\{\psi_j\}, \{\phi_j\} \subset H^1$ weakly $H^1$ converging to $\phi_0$ and $\psi_0$, there holds the following: 
\begin{enumerate}[$(i)$]
\item
$\displaystyle\int_{\R^N}\left(\frac{1}{|\cdot|^{N-\alpha_j}}*(v_{\alpha_j}^{p-1}\phi_j)\right)v_{\alpha_j}^{p-1}\psi_j\,dx \to \left(\int_{\R^N}v_0^{p-1}\phi_0\,dx\right)\left(\int_{\R^N}v_0^{p-1}\psi_0\,dx\right)$;
 
\item
$\displaystyle\int_{\R^N}\left(\frac{1}{|\cdot|^{N-\alpha_j}}*u_{\alpha_j}^p\right)u_{\alpha_j}^{p-2}\phi_j\psi_j\,dx \to \left(\int_{\R^N}u_0^p\,dx\right)\left(\int_{\R^N}u_0^{p-2}\phi_0\psi_0\,dx\right).$
\end{enumerate}
\end{lemma}
\begin{proof}
From the same argument to Lemma \ref{conv-linearized0}, one can see that $v_{\alpha_j}^{p-1}\phi_j$ is compact in $L^1 \cap L^{2^*/p}$, where $2^* = 2N/(N-2)$. 
By following the same argument in the proof $(ii)$ of Proposition \ref{est-identity}, one is able to see
\begin{equation}\label{est-infiniteN}
\left\|\frac{1}{|\cdot|^{N-\alpha}}*f\right\|_{L^\infty} \leq C\|f\|_{L^1 \cap L^{2^*/p}},
\end{equation}
where $f \in L^1 \cap L^{2^*/p}$, $\alpha$ is near $N$ and $C$ is independent of $\alpha$ near $N$.
Using $(ii)$ of Proposition \ref{est-identity}, we then compute
\[
\begin{aligned}
&\int_{\R^N}\left(\frac{1}{|\cdot|^{N-\alpha_j}}*(v_{\alpha_j}^{p-1}\phi_j)\right)v_{\alpha_j}^{p-1}\psi_j\,dx \\
&= \underbrace{\int_{\R^N}\left(\frac{1}{|\cdot|^{N-\alpha_j}}*(v_{\alpha_j}^{p-1}\phi_j-v_0^{p-1}\phi_0)\right)v_{\alpha_j}^{p-1}\psi_j\,dx}_{(A)} \\
&\qquad\qquad\qquad+\int_{\R^N}\left(\frac{1}{|\cdot|^{N-\alpha_j}}*(v_0^{p-1}\phi_0)\right)v_{\alpha_j}^{p-1}\psi_j\,dx \\
&= (A) +\left(\int_{\R^N}v_0^{p-1}\phi_0\,dx\right)\left(\int_{\R^N}v_0^{p-1}\psi_0\,dx\right) +o(1).
\end{aligned}
\]
and the estimate \eqref{est-infiniteN} says
\[
|(A)| \leq C\|v_{\alpha_j}^{p-1}\phi_j-v_0^{p-1}\phi_0\|_{L^1 \cap L^{2^*/p}}\|v_{\alpha_j}^{p-1}\psi_j\|_{L^1} = o(1),
\]
which proves the assertion $(i)$.

To prove $(ii)$, we claim that $v_0^{p-2}\phi_j\psi_j$ is $L^1$ compact. Observe
\[
\int_{\R^N\setminus B_R} |v_0^{p-2}\phi_j\psi_j|\,dx \leq \|v_0\|_{L^p(\R^N\setminus B_R)}^{p-2}\|\phi_j\|_{L^p(\R^N)}\|\psi_j\|_{L^p(\R^N)}
\]
so that $v_0^{p-2}\phi_j\psi_j$ is tight. Also we note that $\|v_0\|_{C^1(\R^N)}$ is finite by the elliptic regularity theory 
and for every $B_R$, there exists $C_R$ such that $\|v_0^{p-3}\|_{B_R} \leq C_R$ because $v_0$ is continuous and positive everywhere. 
Then we can see that $v_0^{p-2}\phi_j\psi_j \in W^{1,1}(B_R)$ from the estimate
\[
\begin{aligned}
&\|v_0^{p-2}\phi_j\psi_j\|_{W^{1,1}(B_R)} \\
&\leq \|v_0^{p-2}\phi_j\psi_j\|_{L^1(B_R)}+ \|v_0^{p-3}\nabla v_0\phi_j\psi_j\|_{L^1(B_R)}
+\|v_0^{p-2}\nabla \phi_j \psi_j\|_{L^1(B_R)}+\|v_0^{p-2}\phi_j\nabla\psi_j\|_{L^1(B_R)} \\
& \leq\|v_0\|_{L^p}^{p-2}\|\phi_j\|_{L^p}\|\psi_j\|_{L^p}+C_R\|\nabla v_0\|_{L^\infty}\|\phi_j\|_{L^2}\|\psi_j\|_{L^2} \\
&\qquad\qquad\qquad+\|v_0\|_{L^\infty}^{p-2}\|\nabla\phi_j\|_{L^2}\|\psi_j\|_{L^2}+\|v_0\|_{L^\infty}^{p-2}\|\phi_j\|_{L^2}\|\nabla\psi_j\|_{L^2} \\
&< \infty
\end{aligned}
\]
so that $v_0^{p-2}\phi_j\psi_j$ is locally $L^1$ compact by the compact Sobolev embedding. 
By combining  this with tightness, we conclude that $v_0^{p-2}\phi_j\psi_j$ is $L^1$ compact 
and consequently, $v_{\alpha_j}^{p-2}\phi_j\psi_j$ is $L^1$ compact. 
Now, the remaining part of the proof follows the same line with the previous one. 
\end{proof}

%\begin{lemma}\label{uniform-boundN}
%For given $p \in (2, 2N/(N-2))$, let $v_\alpha$ be a family of the unique positive radial ground states of \eqref{rescaled-eq}.
%Then one has the uniform $L^\infty$ bound of $v_\alpha$ for $\alpha$ sufficiently close to $N$.
%In particular, $v_\alpha$ has the uniform $L^q$ bound for every $q \in [2, \infty]$ by the interpolation. 
%\end{lemma}
%\begin{proof}
%As in the proof of Lemma \ref{uniform-bound0}, define a function
%\[
%c(x) := 1-\left(\frac{1}{|\cdot|^{N-\alpha}}*v_\alpha^p\right)v_{\alpha}^{p-2}
%\]
%so we need to show there exists $q > N/2$ such that $\|c\|_{L^q}$ is bounded uniformly for $\alpha$ near $N$.
%We select $q = 2^*/(p-2)$. Then $q > N/2$ if and only if $p < 2N/(N-2)$. 
%The uniform $L^q$ bound of $c(x)$ follows from $(ii)$ of Proposition \ref{est-identity}. This completes the proof. 
%\end{proof}

\begin{prop}\label{nondegN}
Let $p \in (2, 2N/(N-2))$ and $v_\alpha$ be a family of unique positive radial ground states of \eqref{rescaled-eq}. Then for $\alpha < N$ sufficiently close to $N$, the linearized equation of \eqref{rescaled-eq} at $v_\alpha$, given by
\begin{equation}\label{linearized-rescaled-eq}
-\Delta \phi +\phi -p(1/|\cdot|^{N-\alpha} * (v_\alpha^{p-1}\phi))v_\alpha^{p-1}-(p-1)(1/|\cdot|^{N-\alpha} * v_\alpha^p)v_{\alpha}^{p-2}\phi  = 0 \quad \text{in } \R^N, 
\end{equation}
only admits solutions of the form
\[
\phi = \sum_{i=1}^N c_i\partial_{x_i}v_\alpha, \qquad c_i \in \R
\]
in the space $L^2(\R^N)$.
\end{prop}
\begin{proof}
The proof of Proposition \ref{nondegN} follows the same line with the proof of Proposition \ref{nondeg0}. 
The one we only need to show is that $L[\phi] \in H^{-1}$ when $\phi$ is a $L^2$ solution to \eqref{linearized-rescaled-eq} and
\[
L[\phi] := p(1/|\cdot|^{N-\alpha} * (v_\alpha^{p-1}\phi))v_\alpha^{p-1}+(p-1)(1/|\cdot|^{N-\alpha} * v_\alpha^p)v_{\alpha}^{p-2}\phi.
\]
For $\psi \in H^1(\R^N)$, we compute from Proposition \ref{est-identity} that
\[
\begin{aligned}
&L[\phi]\psi = \int_{\R^N}p(1/|\cdot|^{N-\alpha} * (v_\alpha^{p-1}\psi))v_\alpha^{p-1}\phi+(p-1)(1/|\cdot|^{N-\alpha} * v_\alpha^p)v_{\alpha}^{p-2}\phi\psi \,dx \\
&\leq pC\|v_\alpha\|_{H^1}^{p-1}\|\psi\|_{H^1}\|v_\alpha\|_{L^{2(p-1)}}^{p-1}\|\phi\|_{L^2}
+(p-1)C\|v_\alpha\|_{H^1}^p\|v_\alpha\|_{L^\infty}^{p-2}\|\phi\|_{L^2}\|\psi\|_{L^2}
\end{aligned}
\]
from which we deduce $L[\phi] \in H^{-1}$.
\end{proof}

Now, we shall end the proof of Theorem \ref{thm1}.
For $2< p < 2N/(N-2)$ and $\alpha < N$ close to $N$, let $\{u_\alpha\}$ be a family of unique positive radial ground states of \eqref{main-eq}
and $\phi_\alpha \in L^2(\R^N)$ be a solution of the linearized equation \eqref{linearized-eq}. 
Then $\phi_\alpha$ is a solution of \eqref{linearized-rescaled-eq} with $v_\alpha = s(N,\alpha,p)u_\alpha$. 
Then Proposition \ref{nondegN} says that $\phi_\alpha$ is a linear combination of $\partial_{x_i}v_\alpha$'s, which is also a linear combination of $\partial_{x_i}u_\alpha$'s.
This completes the proof of Theorem \ref{thm1}.

\bigskip

{\bf Acknowledgements.}
This work was supported by Kyonggi University Research Grant 2016.

\end{document}